\documentclass[12pt, reqno]{amsart}
\usepackage{amsmath, amsthm, amscd, amsfonts, amssymb, graphicx, color}
\usepackage[bookmarksnumbered, plainpages]{hyperref}
\usepackage{dsfont}
\numberwithin{equation}{section}
\input{amssym.def}
\input{amssym.tex}

\begin{document}

\title[Approximation of general 3-variable Jensen $\rho$-functional inequalities in
complex Banach spaces]
{Approximation of general 3-variable Jensen $\rho$-functional inequalities in
complex Banach spaces}
\author[G. Lu]{Gang Lu$^*$}
\address{Gang Lu \newline \indent  Division of Foundational Teaching, Guangzhou College of Technology and Business, Guangzhou 510850, P.R. China
}
\email{lvgang1234@163.com}\vskip 2mm

\author[W.  Sun]{Wenlong Sun}
\address{Wenlong Sun \newline \indent Department of Mathematics, School of Science, ShenYang University of Technology, Shenyang 110870, P.R. China}
\email{179378033@qq.com}\vskip 2mm

\author[H.Qiao]{Hanyue Qiao}
\address{Hanyue Qiao \newline \indent Department of Mathematics,  Yanbian University, Yanji 133001, P.R. China}
\email{1378435807@qq.com}\vskip 2mm

\author[Y. Jin]{Yuanfeng Jin$^*$}
\address{Yuanfeng Jin \newline \indent Department of Mathematics,  Yanbian University, Yanji 133001, P.R. China}
\email{yfkim@ybu.edu.cn}\vskip 2mm

\author[C. Park]{Choonkil Park}
\address{Choonkil Park\newline \indent  Research Institute for Natural Sciences,
Hanyang University, Seoul 04763,   Korea}
\email{baak@hanyang.ac.kr}\vskip 2mm

\begin{abstract}
In this paper,  we introduce and investigate general 3-variable Jensen $\rho$- functional equation, and
  prove the Hyers-Ulam stability of the Jensen functional
equations associated with the  general 3-variable Jensen $\rho$-functional inequalities in complex Banach spaces.
\end{abstract}



\keywords{Jensen functional inequaty; Hyers-Ulam stability; complex Banach space.\\ $^*$Corresponding authors: lvgang1234@163.com (G.Lu), yfkim@ybu.edu.cn (Y. Jin).}

\theoremstyle{definition}
  \newtheorem{df}{Definition}[section]
    \newtheorem{rk}[df]{Remark}
\theoremstyle{plain}
  \newtheorem{lemma}[df]{Lemma}
  \newtheorem{theorem}[df]{Theorem}
  \newtheorem{corollary}[df]{Corollary}
    \newtheorem{proposition}[df]{Proposition}
\newtheorem{example}[df]{Example}
\setcounter{section}{0}

\maketitle

\baselineskip=16pt

\numberwithin{equation}{section}

\vskip .2in

\section{Introduction and preliminaries}

The stability problem of functional equations originated from a
question of Ulam \cite{Ul} concerning the stability of group
homomorphisms.The essence of the problem is, under what condition does there exists a homomorphism near an
approximate homomorphism?  The study of stability for functional
equation arises  from the Ulam's problem. In 1941,
 Hyers \cite{Hy} gave the first affirmative  answer to the
question of Ulam for Banach spaces. His method was called the {\it direct method}. Later,
Hyers' theorem  was generalized by Aoki \cite{A} for additive
mappings and by  Rassias \cite{R3} for linear mappings by
considering an unbounded Cauchy difference.  A
generalization of the Th.M. Rassias theorem was obtained by G\u
avruta \cite{Ga} by replacing the unbounded Cauchy difference by a
general control function in the spirit of Th.M. Rassias' approach.
The stability problems for several functional equations or
inequations have been extensively investigated by a number of
authors (see \cite{AD}--\cite{EGRG},  \cite{HIR}--\cite{LP}, \cite{pra},
\cite{R1}--\cite{XRX}).

The function equations
\begin{eqnarray}\label{eq11}
&&f(x+y+z)+f(x+y-z)-2f(x)-2f(y)=0\\ \label{eq12}
&& f(x+y+z)-f(x-y-z)-2f(y)-2f(z)=0.
\end{eqnarray}
are  called {\it 3-variable Jensen}.
In \cite{LLJX,Park, SJPL}, Lu {\it et al.}  investigated the  3-variable functional inequalities
 and proved their  stability.

In this paper,  we consider the following functional equations
 \begin{eqnarray}\label{eq11}
&&f(x+y+\alpha z)+f(x+y-\alpha z)-2f(x)-2f(y)=0, \\ \label{eq12}
&& f(x+\beta y+\alpha z)-f(x-\alpha z)-\beta f(y)-2f(\alpha z)=0,
\end{eqnarray} where $\beta$ and $\alpha $ are nonzero real numbers. And discuss  the Hyers-Ulam stability of general 3-variable Jensen $\rho$-functional equations associated with functional inequalities in complex Banach spaces.

Throughout this paper, assume that $X$ is  a complex normed vector space with norm $\|\cdot\|$ and that  $Y$ is a complex Banach space.

\section{Hyers-Ulam stability of (\ref{eq11}) }

 In this section, we prove  the Hyers-Ulam stability  of the 3-variable functional inequality
\begin{eqnarray}\label{eq21}
\begin{split}
\;& \|f(x+y+\alpha z)+f(x+y-\alpha z)-2f(x)-2f(y)\|\\
\;& \leq \left\|\rho_1(f(x+y+\alpha z)-f(x+y)-f(\alpha z))\right\|\\\;&
+\left\|\rho_2(f(x+y-\alpha z)+f(-x)+f(\alpha z-y))\right\|
\end{split}
\end{eqnarray}
 in  complex Banach spaces, where $\rho_1$ and $\rho_2$ are  fixed complex numbers with  $|\rho_1|+3|\rho_2|<2$.

\begin{lemma}\label{lm1}
Let $f:X\rightarrow Y$ be a mapping. If it satisfies (\ref{eq21})
for all $x,y,z\in X$, then $f$ is additive.
\end{lemma}

\begin{proof}
Letting $x=y=z=0$  in (\ref{eq21}), we get
$$2\|f(0)\|\leq (|\rho_1|+3|\rho_2|)\|(f(0)\|$$
and thus $f(0)=0$, $|\rho_1|+3|\rho_2|<2$.

Letting $x=y=0$  in (\ref{eq21}), we get
$$\|f(\alpha z)+f(-\alpha z)\|\leq \|\rho_2(f(-\alpha z)+f(\alpha z))\|$$
and so $f(-x)=-f(x)$ for all $x\in X$.

Letting $z=0$ in (\ref{eq21}), we have
\begin{eqnarray*}
\begin{split}
\;& \|2f(x+y)-2f(x)-2f(y)\| \leq
\|\rho_2 (f(x+y)-f(x)-f(y))\|\\ \;&
\end{split}
\end{eqnarray*}
and so $f(x+y)=f(x)+f(y)$ for all $x,y\in X$. Hence $f:X\rightarrow Y$ is additive.
\end{proof}

\begin{corollary}
Let $f: X\rightarrow Y$ be a mapping satisfying
\begin{eqnarray}
\begin{split}
\;& \|f(x+y+\alpha z)+f(x+y-\alpha z)-2f(x)-2f(y)\|\\
\;& = \left\|\rho_1(f(x+y+\alpha z)-f(x+y)-f(\alpha z))\right\|\\\;&
+\left\|\rho_2(f(x+y-\alpha z)+f(-x)+f(\alpha z-y))\right\|
\end{split}
\end{eqnarray}
for all $x,y,z\in X$. Then $f:X\rightarrow Y$ is additive.
\end{corollary}

We prove the Hyers-Ulam stability of the additive functional inequality (\ref{eq21}) in complex Banach spaces.

\begin{theorem}\label{thm23}
Let $f:X\rightarrow Y$ be a mapping. If there is a function
$\varphi:X^3\rightarrow [0,\infty)$ with $\varphi(0,0,0)=0$ such that
\begin{eqnarray}\label{eqn24}
\begin{split}
\;& \|f(x+y+\alpha z)+f(x+y-\alpha z)-2f(x)-2f(y)\|\\
\;& \leq \left\|\rho_1(f(x+y+\alpha z)-f(x+y)-f(\alpha z))\right\|\\\;&
+\left\|\rho_2(f(x+y-\alpha z)+f(-x)+f(\alpha z-y))\right\|+\varphi(x,y,z)
\end{split}
\end{eqnarray}
and
\begin{eqnarray}\label{eqn25}
\lim_{j\rightarrow \infty}\frac{1}{2^{ j}}
\varphi\left(2^{j}x, 2^{j}y,2^{j}z\right)=0
\end{eqnarray}
 \begin{eqnarray}\label{eqn30}
 \widetilde{\varphi}(x):=\sum_{i=0}^{\infty}\frac{1}{2^{i+1}}\frac{1}{(2-|\rho_2|)}\left(\varphi(2^{i}x,2^{i}x,0)
   +\frac{2|\rho_2|}{1-|\rho_2|}\varphi(0,0,2^{i}x)\right)<\infty
   \end{eqnarray}
for all $x,y,z\in X$, then there exists a unique additive mapping $A: X\rightarrow
Y$ such that
\begin{eqnarray}\label{eqn26}
\|f(x)-A(x)\|
 \leq\widetilde{\varphi}(x)
\end{eqnarray}
for all $x\in X$.
\end{theorem}

\begin{proof}
Letting $x=y=z=0$ in (\ref{eqn24}), we get
\begin{eqnarray}\label{eqn24'}
2\|f(0)\|\leq (|\rho_1|+3|\rho_2|)\|(f(0)\|.\end{eqnarray}
So $f(0)=0$.
Letting $x=y=0$  in (\ref{eq21}), we get
$$\|f(\alpha z)+f(-\alpha z)\|\leq \|\rho_2(f(-\alpha z)+f(\alpha z))\|+\varphi(0,0,z)$$
and so $$\|f(z)+f(-z)\|\leq \frac{\varphi\left(0,0,\frac{z}{\alpha}\right)}{1-|\rho_2|}$$ for all $z\in X$.

Letting $y=x$ and $z=0$ in (\ref{eqn24}), we get
\begin{eqnarray}\label{eqn28}
\|2f(2x)-4f(x)\|
\leq |\rho_2|\|f(2x)-2f(x)\|+2|\rho_2|\|f(x)+f(-x)\|+\varphi(x,x,0)
\end{eqnarray}
and so
$$\|f(2x)-2f(x)\|\leq \frac{1}{2-|\rho_2|}\left(\varphi(x,x,0)+\frac{2|\rho_2|}{1-|\rho_2|}\varphi\left(0,0,\frac{x}{\alpha}\right)\right)$$
for all $x\in X$.
 Thus
\begin{eqnarray*}\left\|f(x)-\frac{f(2x)}{2}\right\| \leq \frac{1}{2}\frac{1}{(2-|\rho_2|)}\left(\varphi(x,x,0)+\frac{2|\rho_2|}{1-|\rho_2|}\varphi\left(0,0,\frac{x}{\alpha}\right)\right)
\end{eqnarray*}
for all $x\in X$.
Hence one may have the following formula for  positive
integers  $m, l$ with $m>l$,
 \begin{eqnarray}\label{eqn29}
 \begin{split}
\;& \left\| \frac{1}{2^{l}}f\left(2^l x\right)-\frac{1}{2^{m}}f\left(2^m x\right)\right\|\\
\;&\leq \sum_{i=l}^{m-1}\frac{1}{2^{i+1}}\frac{1}{(2-|\rho_2|)}\left(\varphi(2^{i}x,2^{i}x,0)
   +\frac{2|\rho_2|}{1-|\rho_2|}\varphi\left(0,0,\frac{2^{i}x}{\alpha}\right)\right)
\end{split}
\end{eqnarray}
for all $x\in X$.

It follows from (\ref{eqn30}) that the sequence $\left\{\frac{f(2^kx)}{2^k}
\right\}$ is a Cauchy sequence for
all $x\in  X $. Since $ Y $ is complete, the sequence  $\left\{\frac{f(2^kx)}{2^k}
\right\}$  converges. So one may
define the mapping $A: X\rightarrow Y$ by
$$A(x):=\lim_{k\rightarrow \infty}\left\{\frac{f(2^kx)}{2^k}
\right\}, \quad \forall x\in X.$$
Taking $l=0$ and letting $m$ tend to $\infty$ in (\ref{eqn29}), we get (\ref{eqn26}).

It follows from (\ref{eqn24}) that
\begin{eqnarray}
\begin{split}
\;& \|A(x+y+\alpha z)+A(x+y-\alpha z)-2A(x)-2A(y)\|\\
\;& =\lim_{n\rightarrow \infty}\frac{1}{2^n}\left\|f\left[2^n(x+y+\alpha z)\right]+f\left[2^n(x+y-\alpha z)\right]
-2f\left(2^nx\right)-2f\left(2^ny\right)\right\|\\
\;& \leq \lim_{n\rightarrow \infty}\frac{1}{2^n}\left\|\rho_1\left(f\left[2^n(x+y+\alpha z)\right]-f\left(2^nx+2^ny\right)-
f\left(2^n \alpha z\right)\right)\right\|\\
\;&
+\lim_{n\rightarrow \infty}\frac{1}{2^n}\left\|\rho_2\left(f\left[2^n(x+y-\alpha z)\right]+f\left(-2^nx\right)
+f\left(-2^ny+2^n\alpha z\right)\right)\right\|\\\;&+\lim_{n\rightarrow \infty}
\frac{1}{2^n}\varphi\left(2^nx,2^ny,2^n\alpha z\right) \\
\;& =\|\rho_1(A(x+y+\alpha z)-A(x+y)-A(\alpha z))\|\\\;&+\|\rho_2(A(x+y-\alpha z)+A(-x)+A(-y+\alpha z))\|
\end{split}
\end{eqnarray}
for all $x,y,z\in X$.
One can see that $A$ satisfies the inequality (\ref{eq21}) and so it is additive by Lemma \ref{lm1}.

Now, we show that the uniqueness of $A$. Let $T: X\rightarrow Y$ be
another additive mapping satisfying (\ref{eqn24}). Then one has

\begin{eqnarray*}
\begin{split}
\; &\|A(x)-T(x)\|= \left\|\frac{1}{2^k} A\left(2^kx\right)-\frac{1}{2^k}
T\left(2^k x\right)\right\| \\
 \;&\leq
 \frac{1}{2^{k}}\left(\left\|A\left(2^k x\right)-f\left(2^k x\right)\right\|\right.\\ \; &
 \left.+\left\|T\left(2^k x\right)-f\left(2^k x\right)\right\|\right) \\
 \;& \leq 2\frac{1}{2^{k}}\widetilde{\varphi}(2^k x)=\sum_{i=k}^{\infty}\frac{1}{2^{i+1}}\frac{1}{(2-|\rho_2|)}\left(\varphi(2^{i}x,2^{i}x,0)
   +\frac{2|\rho_2|}{1-|\rho_2|}\varphi\left(0,0,\frac{2^{i}x}{\alpha}\right)\right),
 \end{split}
\end{eqnarray*}
which tends to zero as $k\rightarrow \infty$ for all $x\in X$. So we
can conclude that $A(x)=T(x)$ for all $x\in X$.
\end{proof}

\begin{corollary}
Let $r<1$ and $\theta$ be nonnegative real numbers and $f:X\rightarrow Y$ be a mapping such that
\begin{eqnarray}\label{eqn211}
\begin{split}
\;& \|f(x+y+\alpha z)+f(x+y-\alpha z)-2f(x)-2f(y)\|\\
\;& =\left\|\rho_1(f(x+y+\alpha z)-f(x+y)-f(\alpha z))\right\|\\\;&
+\left\|\rho_2(f(x+y-\alpha z)+f(-x)+f(\alpha z-y))\right\|+\theta(\|x\|^r+\|y\|^r+\|z\|^r)
\end{split}
\end{eqnarray}
for all $x,y,z\in X$. Then there exists a unique additive mapping $A:X\rightarrow Y$ such that
\begin{eqnarray}
\|f(x)-A(x)\|\leq \frac{2\theta}{(2-2^r)}\cdot\frac{1}{(1-|\rho_2|)(2-|\rho_2|)}\|x\|^r
\end{eqnarray}
for all $x\in X$.
\end{corollary}

\begin{theorem}\label{thm25}
Let $f:X\rightarrow Y$ be a mapping with $\varphi(0,0,0)=0$. If there is a
function $\varphi :X^3\rightarrow [0,\infty)$ satisfying (\ref{eqn24}) such that
\begin{eqnarray}
 \lim_{j\rightarrow \infty}2^{j} \varphi \left(\frac{x}{2^j},\frac{y}{2^j},\frac{z}{2^j}\right)=0
\end{eqnarray}
for all $x,y,z\in X$, then there exists a unique additive mapping $A:X\rightarrow
Y$ such that
$$\|f(x)-A(x)\|\leq \widetilde{\varphi}\left(\frac{x}{2}\right):=\sum_{i=0}^{\infty}\frac{1}{2^{i}}\frac{1}{2-|\rho_2|}\left(\varphi(\frac{x}{2^{i+1}},\frac{
x}{2^{i+1}},0)+\frac{2|\rho_2|}{1-|\rho_2|}\varphi(0,0,\frac{x}{\alpha 2^{i+1}})\right)$$
for all $x\in X$.
\end{theorem}

\begin{proof}
Similar to the proof of  Theorem \ref{thm23}, we can get
\begin{eqnarray*}
\left\|f(x)-2f\left(\frac{x}{2}\right)\right\|
\leq\frac{1}{2-|\rho_2|}\left(\varphi(\frac{x}{2},\frac{x}{2},0)
+\frac{2|\rho_2|}{1-|\rho_2|}\varphi(0,0,\frac{x}{2\alpha })\right)
\end{eqnarray*}
for all $x\in X$.

Next, we can prove that the sequence $\{2^nf\left(\frac{x}{2^n}\right)$ is a Cauchy sequence for
all $x\in X$, and define a mapping $A:
X\rightarrow Y$ by
$$A(x):=\lim_{n\rightarrow \infty}2^n f\left(\frac{x}{2^n}\right)$$
for all $x\in X$ .

The rest proof is similar to the corresponding part of the proof of Theorem \ref{thm23}.
\end{proof}

\begin{corollary}
Let $r>1$ and $\theta$ be nonnegative real numbers and $f:X\rightarrow Y$ ba a mapping such that
\begin{eqnarray}\label{eqn211}
\begin{split}
\;& \|f(x+y+\alpha z)+f(x+y-\alpha z)-2f(x)-2f(y)\|\\
\;&\leq \left\|\rho_1(f(x+y+\alpha z)-f(x+y)-f(\alpha z))\right\|\\\;&
+\left\|\rho_2(f(x+y-\alpha z)+f(-x)+f(\alpha z-y))\right\|+\theta(\|x\|^r+\|y\|^r+\|z\|^r)
\end{split}
\end{eqnarray}
for all $x,y,z\in X$. Then there exists a unique additive mapping $A:X\rightarrow Y$ such that
\begin{eqnarray}
\|f(x)-A(x)\|\leq \frac{2^{1+r}\theta}{2^r-1}\frac{1}{(1-|\rho_{2}|)(2-|\rho_{2}|)}\|x\|^r
\end{eqnarray}
for all $x\in X$.
\end{corollary}

\section{Hyers-Ulam stability of (\ref{eq12}) }

In this section, we prove that the Hyers-Ulam stability  of the 3-variable functional inequality
\begin{eqnarray}\label{eqn31}
\begin{split}
\;& \|f(x+\beta y+\alpha z)-f(x-\alpha z)-\beta f(y)-2f(\alpha z)\|\\
\;& \leq \left\|\rho_1(f(x+\alpha z)-f(x)-f(\alpha z))\right\|\\\;&
+\left\|\rho_2(f(x+\beta y-\alpha z)-f(x)-\beta f(y)+f(\alpha z))\right\|
\end{split}
\end{eqnarray}
 in  complex Banach space,  where $\rho_1$ and $\rho_2$ are  fixed complex numbers with $|\rho_2|<1$ and $|\beta +2|\geq |\rho_1|+|\rho_2(1-\beta )|$.

\begin{lemma}\label{lm31}
Let $f:X\rightarrow Y$ be a mapping. If it satisfies (\ref{eqn31})
for all $x,y,z\in X$, then $f$ is additive.
\end{lemma}

\begin{proof}
Letting $x=y=z=0$ in (\ref{eqn31}) for all $x,y,z\in X$, we
get
\begin{eqnarray}
\|(\beta +2)f(0)\|\leq (|\rho_1|+|\rho_2||\beta -1|)\| f(0)\|.
\end{eqnarray}
Thus $f(0)=0$.

Letting $x=y=0$ in (\ref{eqn31}), we get
\begin{eqnarray*}(1-|\rho_2|)\|f(\alpha z)+f(-\alpha z)\|\leq 0\end{eqnarray*}
and so $f(-x)=-f(x)$ for all $x\in X$.

Letting $x=0$ in (\ref{eqn31}), we have
\begin{eqnarray}\label{eqn33}
\begin{split}
 \|f(\beta y+\alpha z)-f(\alpha z)-\beta f(y)\| \leq
\|\rho_2 (f(\beta y-\alpha z)-\beta f(y)+f(\alpha z))\|
\end{split}
\end{eqnarray}
for all $y,z\in X$.

Letting $z=-z$ in (\ref{eqn33}), we get
\begin{eqnarray}
\|f(\beta y-\alpha z)+f(\alpha z)-\beta f(y)\|\leq |\rho_2|\|f(\beta y+\alpha x)-\beta f(y)-f(\alpha z)\|
\end{eqnarray}
for all $y,z\in X$.
Thus
\begin{eqnarray}\label{eqn31'}
\|f(\beta y-\alpha z)-\beta f(y)+f(\alpha z)\|\leq 0
\end{eqnarray}
and so
$$\|f(y+z)-f(y)-f(z)\|= 0$$
for all $y,z\in X$.
Hence $f:X\rightarrow Y$ is additive.
\end{proof}

\begin{corollary}
Let $f: X\rightarrow Y$ be a mapping satisfying
\begin{eqnarray}\label{eqn36}
\begin{split}
\;& \|f(x+\beta y+\alpha z)-f(x-\alpha z)-\beta f(y)-2f(\alpha z)\|\\
\;& =\left\|\rho_1(f(x+\alpha z)-f(x)-f(\alpha z))\right\|\\\;&
+\left\|\rho_2(f(x+\beta y-\alpha z)-f(x)-\beta f(y)+f(\alpha z))\right\|
\end{split}
\end{eqnarray}
for all $x,y,z\in X$. Then $f:X\rightarrow Y$ is additive.
\end{corollary}

We prove the Hyers-Ulam stability of the  functional inequality (\ref{eqn31}) in complex Banach spaces.

\begin{theorem}\label{thm33}
Let $f:X\rightarrow Y$ be a mapping. Assume that  there is a function
$\varphi:X^3\rightarrow [0,\infty)$ with $\varphi(0,0,0)=0$ such that
\begin{eqnarray}\label{eqn34}
\begin{split}
\;& \|f(x+\beta y+\alpha z)-f(x-\alpha z)-\beta f(y)-2f(\alpha z)\|\\
\;& \leq\left\|\rho_1(f(x+\alpha z)-f(x)-f(\alpha z))\right\|\\\;&
+\left\|\rho_2(f(x+\beta y-\alpha z)-f(x)-\beta f(y)+f(\alpha z))\right\|+\varphi(x,y,z)
\end{split}
\end{eqnarray}
and
\begin{eqnarray}\label{eqn35}
\lim_{j\rightarrow \infty}\frac{1}{|1+\beta|^{ j}}
\varphi\left((1+\beta)^{j}x, (1+\beta)^{j}y,(1+\beta)^{j}z\right)=0
\end{eqnarray}
for all $x,y,z\in X$. Then there exists a unique additive mapping $A:X\rightarrow
Y$ such that
\begin{eqnarray}\label{eqn36}
\|f(x)-A(x)\|
  \leq \widetilde{\varphi}(x,x,0)
\end{eqnarray}
for all $x\in X$, where
\begin{eqnarray}\label{eqn37}
\widetilde{\varphi}(x,y,z):=\frac{1}{|1+\beta|(1-|\rho_1|)}
\sum_{j=0}^\infty\frac{1}{|1+\beta|^{ j}}
\varphi\left((1+\beta)^{j}x, 1+\beta)^{j}y,1+\beta)^{j}z\right)<\infty
\end{eqnarray}
for all $x, y, z\in X$.
\end{theorem}

\begin{proof}
Letting $x=y=z=0$ in (\ref{eqn34}), we get
\begin{eqnarray}\label{eqn24'}
\|(\beta +2)f(0)\|\leq (|\rho_1|+|(1-\beta)\rho_2|)\| f(0)\|.\end{eqnarray}
So $f(0)=0$.

Letting $z=0$ and $y=x$ in (\ref{eqn34}), we get
\begin{eqnarray}\label{eqn312}
\|f((1+\beta)x)-(1+\beta)f(x)\|
\leq |\rho_2|\|f((1+\beta)x)-(1+\beta)f(x)\|+\varphi(x,x,0)
\end{eqnarray}
for all $x\in X$.

 Thus
\begin{eqnarray*}
\left\|f(x)-\frac{f((1+\beta)x)}{1+\beta}\right\|
 \leq \frac{1}{1-|\rho_2|}
\frac{1}{|1+\beta|}\varphi\left(x,x,0\right)
\end{eqnarray*}
for all $x\in X$ .

Hence one may have the following formula for  positive
integers  $m, l$ with $m>l$,
\begin{eqnarray}\label{eqn39}
\begin{split}
\;&\left\| \frac{1}{|1+\beta|^{l}}f\left((1+\beta)^l x\right)-\frac{1}{|1+\beta|^{m}}f\left((1+\beta)^m x\right)\right\|\\
   \;& \leq \frac{1}{|1+\beta|(1-|\rho_1|)} \sum_{i=l}^{m-1}\frac{1}{|1+\beta|^{i}}\varphi\left(|1+\beta|^i x,|1+\beta|^i x, 0\right),
\end{split}\end{eqnarray}
for all $x\in X$.

It follows from (\ref{eqn37}) that the sequence $\left\{\frac{f(1+\beta)^kx)}{(1+\beta)^k}
\right\}$ is a Cauchy sequence for
all $x\in  X $. Since $ Y $ is complete, the sequence  $\left\{\frac{f((1+\beta)^kx)}{(1+\beta)^k}
\right\}$  converges. So one may
define the mapping $A: X\rightarrow Y$ by
$$A(x):=\lim_{k\rightarrow \infty}\left\{\frac{f((1+\beta)^kx)}{(1+\beta)^k}
\right\}, \quad \forall x\in X.$$
Taking $m=0$ and letting $l$ tend to $\infty$ in (\ref{eqn39}), we get (\ref{eqn36}).

It follows from (\ref{eqn34}) that
\begin{eqnarray}
\begin{split}
\;& \|A(x+\beta y+\alpha z)-A(x-\alpha z)-\beta A(y)-2A(\alpha z)\|\\
\;& =\lim_{n\rightarrow \infty}\frac{1}{|1+\beta|^n}
\left\|f\left[(1+\beta)^n(x+\beta y+\alpha z)\right]+
f\left[(1+\beta)^n(x-\alpha z)\right]\right.\\
\;&\left.
-\beta f\left((1+\beta)^ny\right)-2f\left((1+\beta)^n \alpha z\right)\right\|\\
\;& \leq \lim_{n\rightarrow \infty}\frac{1}{|1+\beta|^n}
\left\|\rho_1\left(f\left[(1+\beta)^n(x+\alpha z)\right]-
f\left[(1+\beta)^n\left(x\right)\right]-f\left((1+\beta)^n\alpha z\right)\right)\right\|\\
\;&
+\lim_{n\rightarrow \infty}\frac{1}{|1+\beta|^n}
\left\|\rho_2\left(f\left[(1+\beta)^n(x+\beta y-\alpha z)\right]-f\left((1+\beta)^nx\right)\right.\right.\\
\;&\left.\left.
-\beta f\left((1+\beta)^ny\right)+f\left((1+\beta)^n \alpha z\right)\right)\right\|\\\;&+\lim_{n\rightarrow \infty}
\frac{1}{|1+\beta|^n}\varphi\left((1+\beta)^nx,(1+\beta)^ny,(1+\beta)^nz\right) \\
\;& =\|\rho_1(A(x+\alpha z)-A(x)-A(\alpha z))\|\\\;&+\|\rho_2(A(x+\beta y-\alpha z)-A(x)-\beta A(y)+A(\alpha z))\|
\end{split}
\end{eqnarray}
for all $x,y,z\in X$.
One can see that $A$ satisfies the inequality (\ref{eqn31}) and so it is additive by Lemma \ref{lm31}.

Now, we show that the uniqueness of $A$. Let $T: X\rightarrow Y$ be
another additive mapping satisfying (\ref{eqn34}). Then one has

\begin{eqnarray*}
\begin{split}
\; &\|A(x)-T(x)\|= \left\|\frac{1}{(1+\beta)^k} A\left((1+\beta)^kx\right)-\frac{1}{(1+\beta)^k}
T\left((1+\beta)^k x\right)\right\| \\
 \;&\leq
 \frac{1}{|1+\beta|^{k}}\left(\left\|A\left((1+\beta)^k x\right)-f\left((1+\beta)^k x\right)\right\|\right.\\ \; &
 \left.+\left\|T\left((1+\beta)^k x\right)-f\left((1+\beta)^k x\right)\right\|\right) \\
 \;& \leq 2\frac{1}{|1+\beta|^{k}}\widetilde{\varphi}(x,x,0),
 \end{split}
\end{eqnarray*}
which tends to zero as $k\rightarrow \infty$ for all $x\in X$. So we
can conclude that $A(x)=T(x)$ for all $x\in X$.
\end{proof}

\begin{corollary}
Let $r>1$ and $\theta$ be nonnegative real numbers and $f:X\rightarrow Y$ ba a mapping such that
\begin{eqnarray}\label{eqn211}
\begin{split}
\;& \|f(x+\beta y+\alpha z)-f(x-\alpha z)-\beta f(y)-2f(\alpha z)\|\\
\;& \leq \left\|\rho_1(f(x+\alpha z)-f(x)-f(\alpha z))\right\|\\\;&
+\left\|\rho_2(f(x+\beta y-\alpha z)-f(x)-\beta f(y)+f(\alpha z))\right\|+\theta(\|x\|^r+\|y\|^r+\|z\|^r)
\end{split}
\end{eqnarray}
for all $x,y,z\in X$ with $|1+\beta|>1$. Then there exists a unique additive mapping $A:X\rightarrow Y$ such that
\begin{eqnarray}
\|f(x)-A(x)\|\leq \frac{2\theta}{|1+\beta|-|1+\beta|^r}\frac{1}{1-|\rho_{2}|}\|x\|^r
\end{eqnarray}
for all $x\in X$.
\end{corollary}

\begin{theorem}\label{thm35}
Let $f:X\rightarrow Y$ be a mapping with $f(0)=0$. If there is a
function $\varphi :X^3\rightarrow [0,\infty)$ satisfying (\ref{eqn34}) such that
\begin{eqnarray}
\widetilde{\varphi}(x,y,z):=\sum_{j=1}^\infty |1+\beta|^{j} \varphi \left(\frac{x}{(1+\beta)^j},\frac{y}{(1+\beta)^j},\frac{z}{(1+\beta)^j}\right)<\infty
\end{eqnarray}
for all $x,y,z\in X$, then there exists a unique additive mapping $A:X\rightarrow
Y$ such that
\begin{eqnarray}
\|f(x)-A(x)\|\leq\frac{1}{1-|\rho_2|} \widetilde{\varphi}\left(\frac{x}{1+\beta },\frac{x}{1+\beta},0\right)
\end{eqnarray}
for all $x\in X$.
\end{theorem}
\begin{proof}
Similar to the proof of Theorem \ref{thm33}, we can get
\begin{eqnarray*}
\left\|f(x)-(1+\beta)f\left(\frac{x}{1+\beta}\right)\right\|
\leq \frac{1}{1-|\rho_2|}\varphi\left(\frac{x}{1+\beta},\frac{x}{1+\beta},0\right)
\end{eqnarray*}
for all $x\in X$.

Next, we can prove that the sequence $\{(1+\beta)^nf\left(\frac{x}{(1+\beta)^n}\right)\}
$ is a Cauchy sequence for
all $x\in X$ and define a mapping $A:
X\rightarrow Y$ by
$$A(x):=\lim_{n\rightarrow \infty}(1+\beta)^n f\left(\frac{x}{(1+\beta)^n}\right)$$
for all $x\in X$. The rest of the proof   is similar to the corresponding part of the proof of Theorem \ref{thm33}.
\end{proof}

\begin{corollary}
Let $r>1$ and $\theta$ be nonnegative real numbers and $f:X\rightarrow Y$ ba a mapping such that
\begin{eqnarray}
\begin{split}
\;& \|f(x+\beta y+\alpha z)-f(x-\alpha z)-\beta f(y)-2f(\alpha z)\|\\
\;& \leq \left\|\rho_1(f(x+\alpha z)-f(x)-f(\alpha z))\right\|\\\;&
+\left\|\rho_2(f(x+\beta y-\alpha z)-f(x)-\beta f(y)+f(\alpha z))\right\|+\theta(\|x\|^r+\|y\|^r+\|z\|^r)
\end{split}
\end{eqnarray}
for all $x,y,z\in X$ and $|1+\beta|<1$. Then there exists a unique additive mapping $A:X\rightarrow Y$ such that
\begin{eqnarray}
\|f(x)-A(x)\|\leq \frac{2\theta}{|1+\beta|^r-|1+\beta|}\frac{1}{1-|\rho_{2}|}\|x\|^r
\end{eqnarray}
for all $x\in X$.
\end{corollary}

\section*{Competing interests}

The author declares that he  has no competing interests.

\section*{Authors' contributions}

The author conceived of the study, participated in its design and
coordination, drafted the manuscript, participated in the sequence
alignment, and read and approved the final manuscript.

\section*{Funding}

This work was supported by National Natural Science Foundation of China  (No.  11761074), the Projection of the Department of Science and Technology of JiLin Province and the Education Department of Jilin Province (No. 20170101052JC) and  the scientific research project of Guangzhou College of Technology and Business in 2020(No. KA202032).
\bigskip
\bibliographystyle{amsplain}

\end{document}